\documentclass[11pt]{article}
\usepackage[english,activeacute]{babel}
\usepackage{amsmath,amsfonts,amssymb,amstext,amsthm,amscd,mathrsfs,amsbsy,xypic,centernot}
\usepackage{xcolor}
\usepackage{MnSymbol}
\usepackage{graphicx}
\usepackage[all]{xy}

\newtheorem{teo}{Theorem}[section]
\newtheorem{pro}[teo]{Proposition}
\newtheorem{coro}[teo]{Corollary}
\newtheorem{lem}[teo]{Lemma}

\theoremstyle{definition}
\newtheorem{defi}[teo]{Definition}
\newtheorem{exam}[teo]{Example}
\newtheorem{rem}[teo]{Remark}

\newcommand{\N}{\mathbb N}

\newcommand{\K}{\mathbb K}
\newcommand{\pols}{\mathbb K[x_1,\ldots,x_s]}

\newcommand{\A}{\mathbb A}
\newcommand{\F}{\mathcal{F}}
\newcommand{\V}{\mathbf{V}}
\newcommand{\m}{\mathfrak{m}}

\newcommand{\J}{\mathcal J}

\newcommand{\Ome}{\Omega^{1}_{A/\K}}
\newcommand{\Omen}{\Omega^{1}_{A_n/\K}}
\newcommand{\OmeA}{\Omega^{(m)}_{A/\K}}

\newcommand{\Ls}{\Lambda}
\newcommand{\Lsc}{\Lambda^0}
\newcommand{\cc}{\underline{0}}
\newcommand{\Dnm}{D_n(\Jac_m(f))}
\newcommand{\OB}{\OmeA\otimes_A B_n}

\newcommand{\Jac}{\operatorname{Jac}}

\newcommand{\HS}{\operatorname{HS}}

\newcommand{\pd}{\operatorname{projdim}}
\newcommand{\Sp}{\operatorname{Spec}}
\newcommand{\rk}{\operatorname{rank}}
\newcommand{\Sing}{\operatorname{Sing}}

\providecommand{\msc}[1]{{\textbf{MSC 2020:}} #1}

\begin{document}

\title{A Nobile-like theorem for jet schemes of hypersurfaces}
\author{Paul Barajas, Daniel Duarte}
\maketitle

\begin{abstract}
We prove that, for the jet scheme of a singular hypersurface, the blowup of a certain jet-related module is not an isomorphism. In conjunction with recent developments in the theory of Nash blowups, our result holds over fields of arbitrary characteristic. Our approach is based on explicit presentations given by a higher-order Jacobian matrix combined with a certain jet-related matrix.
\end{abstract}


\msc{14F10,14B05,13C13}


\section*{Introduction}

In recent years, several authors have proposed higher-order versions of the Jacobian matrix. Initially, such a matrix was introduced as an attempt to give an explicit presentation of the module of high-order differentials. That presentation was used to study higher Nash blowups of hypersurfaces \cite{D2}. Later on, the same matrix reappeared in \cite{BD,BJNB}. In those papers, the matrix was used to investigate singularities in positive characteristic and algebraic properties of the module of high-order differentials. Moreover, a higher-order Jacobian matrix of a morphism was used to solve a conjecture concerning the higher Nash blowup of curves \cite{ChDG}. Finally, a further application of this matrix was given in \cite{Ch}, in relation to the higher Nash blowup of the $A_n$-singularity.

In another but related direction, in \cite{dFDo2} it is described a matrix associated to the induced map on jet schemes of the Nash transformation of a coherent sheaf. The first goal of this paper is to further develop the study of this jet-related matrix by combining it with the higher-order Jacobian matrix. Firstly, we investigate some basic homological properties of the module corresponding to the mentioned coherent sheaf, by means of explicit presentations. Secondly, we apply these properties in the particular cases of the module of differentials and the module of high-order differentials. Finally, we give a geometric application of these results regarding Nash blowups.

Recall that the Nash blowup and the higher Nash blowup of an algebraic variety are modifications that replace singular points by limits of tangent spaces or limits of infinitesimal neighborhoods. It has been proposed to solve singularities using these constructions \cite{S,No,Y}. These questions have been extensively studied \cite{No,R,GS1,GS2,Hi,Sp,GT,GM,D0,Cha,Y,T,ChDG,DN2}.

A basic property of Nash blowups, proved by A. Nobile, says that the Nash blowup of a variety is an isomorphism only if the variety is non-singular, in characteristic zero \cite{No}. This theorem was recently revisited in \cite{DN1}, showing the analogous statement for normal varieties in positive characteristic. In addition, there are several higher-order versions of Nobile's theorem (for zero and positive characteristics): for normal toric varieties \cite{D1,DN2}, for toric curves \cite{ChDG}, for normal hypersurfaces \cite{D2,DN1}, for $F$-pure varieties and quotient varieties \cite{DN1}. In additon to being an important result for the theory of Nash blowups, Nobile's theorem has other applications. For instance, it appears in the study of link theoretic characterization of smoothness \cite{dFT}. 

In this paper we present a Nobile-like theorem in the context of jet schemes of hypersurfaces. Our main theorem states that, for the jet scheme of a singular hypersurface, the blowup of a certain jet-related module is not an isomorphism (see Theorem \ref{Nobile}). Following recent developments in the theory of Nash blowups, our result holds in arbitrary characteristic. Finally, the main theorem suggests that this blowup could have some interest in the problem of resolution of singularities of jet schemes.


\section{Nash blowups relative to sheafs and jet schemes}\label{DFDO}

\textit{All rings are assumed to be commutative and with unit element. Letters $n$ and $m$ always denote natural numbers, where $n\geq0$ and $m\geq1$.}

In this section we briefly recall the main concepts that we use throughout this paper: Hasse-Schmidt derivations and algebra, jet schemes, and the Nash transformation relative to a coherent sheaf (our main references on these topics are \cite{V} and \cite{OZ}). We also present a result of T. de Fernex and R. Docampo that was the main motivation for this paper.

Let $\K$ be a ring. Let $A$ be a $\K$-algebra with structural morphism $f:\K\to A$. Let $B$ be another $\K$-algebra. A \textit{Hasse-Schmidt derivation of order $n$ from $A$ to $B$ over $\K$} is a sequence $(D_0,\ldots,D_n)$, where $D_0:A\to B$ is a $\K$-algebra homomorphism, $D_i:A\to B$ is additive and $D_i(f(\lambda))=0$ for $i\in\{1,\ldots,n\}$ and $\lambda\in \K$, and the following formula is satisfied for all $i$:
\[D_i(xy)=\sum_{j+k=i}D_j(x)D_k(y).\]

Now define a $\K$-algebra as the quotient of a polynomial algebra
$$\HS_{A/\K}^n:=A[x^{(i)}| x\in A, i=1,\ldots,n]/I,$$  
where $I$ is the ideal generated by $\{(x+y)^{(i)}-x^{(i)}-y^{(i)}|x,y\in A, i=1,\ldots,n\}$, $\{f(\lambda)^{(i)}|\lambda\in \K, i=1,\ldots,n\}$, and $\{(xy)^{(i)}-\sum_{j+k=i}x^{(j)}y^{(k)}| x,y\in A, i=0,\ldots,n\}$. For $i=0,\ldots, n$, let us define $d_i:A\to \HS^n_{A/\K}$, $d_i(x)=x^{(i)}$. The algebra $\HS_{A/\K}^n$ is called the \textit{Hasse-Schmidt algebra of $A$} and the sequence $(d_0,\ldots,d_n)$ is called the \textit{universal Hasse-Schmidt derivation}. Throughout this paper, we use indistinctly $d_i(x)$ or $x^{(i)}$ to represent the image of $x$ by $d_i$.

Following \cite{V}, we define the $n$-th jet scheme of an affine scheme using the Hasse-Schmidt algebra. Let $\K$ and $A$ be as before, we denote $$J_n(\Sp(A)/\Sp(\K)):=\Sp(\HS^n_{A/\K}).$$
In general, given $X\to Y$ a morphism of schemes, the $n$-th jet scheme of $X$ over $Y$ is defined as $J_n(X/Y):=\mathbf{Spec}(\HS_{X/Y}^n)$, where $\mathbf{Spec}$ is the relative spectrum and the sheaf $\HS^n_{X/Y}$ is constructed using localizations properties of the Hasse-Schmidt algebra. If the base scheme is affine, we simply write $J_n(X)$.

Now we recall the definition of Nash transformation relative to a sheaf. Let $X$ be a Noetherian integral scheme and $\mathcal{F}$ a coherent $\mathcal{O}_X$-module, locally free of constant rank $d$ over an open dense subset $U$ of $X$. Let $G:=\mbox{Grass}_d(\mathcal{F})$ be the Grassmanian of locally free quotients of $\mathcal{F}$ and $\pi:G\to X$ be the structural morphism \cite[Th\'eor\`eme 9.7.4]{Gr}. By considering the section $\sigma:U\to G$ induced by $\pi$, the \textit{Nash trasformation of $X$ relative to $\mathcal{F}$} is defined as the closure $N(X,\mathcal{F}):=\overline{\sigma(U)}$ with its reduced structure \cite[Definition 1.1]{OZ}.

Particular cases of this construction are the Nash blowup and the higher Nash blowup of an algebraic variety.

\begin{defi}\cite{S,No,OZ,Y}\label{higher Nash}
Let $X$ be a variety over an algebraically closed field $\K$.
\begin{itemize}
\item We call $N(X,\Omega^1_{X/\K})$ the \textit{Nash blowup of} $X$, where $\Omega^1_{X/\K}$ is the sheaf of K\"ahler differentials of $X$ over $\K$.
\item We call $N(X,\Omega^{(m)}_{X/\K})$ the \textit{higher Nash blowup of} $X$, where $\Omega^{(m)}_{X/\K}$ is the sheaf of high-order K\"ahler differentials (see section \ref{omegam} below). 
\end{itemize}
\end{defi}

A morphism of schemes induces a morphism on jet schemes. However, the induced morphism not always share the properties of the original one. For instance, projectivity is not preserved \cite[Example 5.12]{V}. Hence, it is natural to ask how to projectivize the induced morphism. This question is answered in the following theorem, for main components.

\begin{teo}\cite[Theorem 1.2]{dFDo2}\label{FD Nash}
Let $X$ be a variety over a field $\K$, and let $\mu:N(X,\F)\to X$ be the Nash transformation of a coherent sheaf $\F$ on $X$. Denote as $J'_n(X)$ the main component of $J_n(X)$ and $\Delta_n=\Sp(\K[t]/\langle t^{n+1} \rangle)$. Let
$$\xymatrix{J_n'(X)\times \Delta_n \ar@{->}[r]^{\gamma_n'} \ar@{->}[d]_{\rho_n'} & X  \\  J_n'(X)}$$
be the diagram induced by restriction from the universal $n$-jet of $X$. Define $\F_n':=(\rho'_n)_{*}(\gamma'_n)^{*}\F$. Then the induced map $\mu'_n:J'_n(N(X,\F))\to J'_n(X)$ factors as
$$\xymatrix{J'_n(N(X,\F))\ar[r]^{\iota_n}&N(J'_n(X),\F'_n)\ar[r]^{\nu_n}&J'_n(X)},$$
where $\iota_n$ is an open immersion and $\nu_n$ is the Nash transformation of $\F'_n$.
\end{teo}

In view of the importance of the sheaf $\F_n'$, in the following section we study basic algebraic properties of $F\otimes_A\HS^n_{A/\K}[t]/\langle t^{n+1} \rangle$, where $F$ is an arbitrary $A$-module.


\section{Presentation of $F\otimes_A B_n$}\label{main}

As in \cite{dFDo1}, we use the following notation from now on: for a ring $\K$ and a $\K$-algebra $A$, let $A_n:=\HS^n_{A/\K}$ and $B_n:=\frac{A_n[t]}{\langle t^{n+1} \rangle}$. Consider the $\K$-algebra homomorphism 
\begin{align}
\gamma^{\#}_n&:A\longrightarrow B_n, \mbox{ }\mbox{ }a\mapsto\sum_{i=0}^nd_i(a)t^i.\notag
\end{align}
Give $B_n$ structure of $A$-module via $\gamma^{\#}_n$. Notice that $B_n$ is also naturally an $A_n$-module. These two module structures are used constantly in what follows.

Let $F$ be an $A$-module. Consider $F\otimes_A B_n$ with the structure of $A_n$-module induced by that of $B_n$. Suppose one has a presentation of $F$ as an $A$-module, given by the transpose of some matrix $L\in\mbox{Mat}_{b\times a}(A)$:
\begin{equation}
\xymatrix{A^{b}\ar[r]^{L^t}&A^{a}\ar[r]&F\ar[r]&0}.\notag
\end{equation}
Consider the following matrix:
\begin{equation}\notag
D_n(L):=
\left( \begin{array}{ccccccc} 
d_0(L) & d_1(L) & d_2(L) & \cdots & d_{n-1}(L) & d_n(L) \\
0  & d_0(L) & d_1(L) & \cdots & d_{n-2}(L) &d_{n-1}(L) \\
0 & 0 & d_0(L) & \cdots & d_{n-3}(L) & d_{n-2}(L)\\
\vdots & \vdots & \vdots & \ddots & \vdots & \vdots&\vdots\\
0 & 0 & \cdots & 0 & d_{0}(L) & d_{1}(L)\\
0 & 0 & \cdots & 0 & 0 & d_{0}(L)
\end{array} \right),
\end{equation}
where $d_i(L)$ is obtained by applying $d_i$ to each entry of $L$. Notice that $D_n(L)$ is a $(n+1)b\times(n+1)a-$matrix.

\begin{rem}\label{Ln}
A matrix having the same shape as $D_n(L)$ was defined by T. de Fernex and R. Docampo, in the case $F$ is torsion-free and $A$ is a domain \cite[Section 3]{dFDo2}. Neither of these is a restriction in this paper.
\end{rem}

We need the following elementary remark for the proof of the main theorem of this section.

\begin{rem}\label{isom An}
Consider the following series of isomorphisms of $A_n$-modules: 
\begin{align}\label{iso tens}
&A^{\lambda}\otimes_A B_n \longrightarrow B_n^{\lambda} \longrightarrow (A_n^{n+1})^{\lambda} \longrightarrow (A_n^{\lambda})^{n+1},\\
&e_k\otimes t^i \mapsto (0,\ldots,t^i,\ldots,0) \mapsto e'_{k,i} \mapsto e''_{i,k},\notag
\end{align}
where $\{e_1,\ldots,e_{\lambda}\}$, $\{e'_0,\ldots,e'_{n}\}$, and $\{e''_1,\ldots,e''_{\lambda}\}$ are the canonical basis of $A^{\lambda}$,  $A_n^{n+1}$, and $A_n^{\lambda}$, respectively; $e'_{k,i}:=(\cc,\ldots,e'_{i},\ldots,\cc)$, where $e_i'$ is placed at the $k$th entry of $(A_n^{n+1})^{\lambda}$; $e''_{i,k}:=(\cc,\ldots,e''_k,\ldots,\cc)$, where $e''_k$ is placed at the $i$th entry of $(A_n^{\lambda})^{n+1}$.

In addition, these isomorphisms induce on $A_n^{n+1}$, $(A_n^{n+1})^{\lambda}$, and $(A_n^{\lambda})^{n+1}$ the $A$-module structure of $B_n$. More precisely, $A$ acts on $(A_n^{\lambda})^{n+1}$ as follows,
\begin{align}\label{struct A mod}
\Big(a,\big((z_{10},\ldots,z_{\lambda0}),\ldots,(z_{1n},\ldots,z_{\lambda n})\big)\Big)\mapsto (P_0,\ldots,P_n),
\end{align}
where $P_i=(d_0(a)z_{1i}+\cdots+d_i(a)z_{10},\ldots,d_0(a)z_{\lambda i}+\cdots+d_i(a)z_{\lambda 0})$.
\end{rem}

\begin{teo}\label{presentation}
Consider a presentation of an $A$-module $F$
$$\xymatrix{A^{b}\ar[r]^{L^t}&A^{a}\ar[r]^{\theta}&F\ar[r]&0}.$$
Then the $A_n$-module $F\otimes_A B_n$ has the following presentation,
\[\xymatrix{\big(A_n^{b}\big)^{n+1}\ar[rr]^{D_n(L)^t}&&\big(A_n^{a}\big)^{n+1} \ar[r]&F\otimes_A B_n\ar[r]&0}.\]
\end{teo}

The proof of this theorem relies on basic facts of the tensor product along with a careful study of all involved module structures. Even though it is elementary, it might be troublesome to follow because of the many indices. In an attempt to make it clearer, we first illustrate the main steps of the proof in the case $a=2$, $b=1$, and $n=2$. 

\begin{exam}
Tensoring the presentation of $F$ with $B_2$ we obtain
\begin{equation}
\xymatrix{A\otimes_A B_2\ar[r]^{\overline{\Delta}}&A^{2}\otimes_AB_2\ar[r]&F\otimes_AB_2\ar[r]&0},\notag
\end{equation}
where $\overline{\Delta}=L^t\otimes Id$.
Using the isomorphisms (\ref{iso tens}) we get:
\begin{equation}
\xymatrix{\big(A_2\big)^3\ar[r]^{\Delta}&\big(A_2^2\big)^3\ar[r]&F\otimes_AB_2\ar[r]&0}.\notag
\end{equation}
We show that $\Delta=D_2(L)^t$.

Denote as $\Phi_1:A\otimes_AB_2\to (A_2)^{3}$ and $\Phi_2:A^{2}\otimes_A B_2\to (A_2^{2})^{3}$ the isomorphisms (\ref{iso tens}). Let $\{1\otimes t^i\}_{0\leq i\leq2}$ and $\{f_{k}\otimes t^i\}_{1\leq k\leq 2,0\leq i\leq 2}$ be the canonical bases of $A\otimes_AB_2$ and $A^{2}\otimes_A B_2$, respectively. Similarly, $\{e_{i}\}_{0\leq i \leq 2}$ and $\{f_{i,k}\}_{0\leq i\leq 2,1\leq k\leq2}$ for  $(A_2)^{3}$ and $(A_2^{2})^{3}$, respectively.  By (\ref{iso tens}), we have $\Phi_1(1\otimes t^i)=e_{i}$ and $\Phi_2(f_{k}\otimes t^i)=f_{i,k}$.

Let $L=(l_1\mbox{ } l_2)$. By construction, $\Delta\circ\Phi_1=\Phi_2\circ\overline{\Delta}$. For $i\in\{0,1,2\}$ we have,
$$\Delta(e_{i})=\Delta(\Phi_1(1\otimes t^i))=\Phi_2(\overline{\Delta}(1\otimes t^i))=\Phi_2(L\otimes t^i)=l_1\cdot f_{i,1}+l_2\cdot f_{i,2}.$$
Now let us describe the products $l_k\cdot f_{i,k}$. By (\ref{struct A mod}), we obtain:
\begin{align} 
l_k\cdot f_{0,k}&=l_k\cdot(f_k,(0,0),(0,0))=d_0(l_k)f_{0,k}+d_1(l_k)f_{1,k}+d_{2}(l_k)f_{2,k},\notag\\
l_k\cdot f_{1,k}&=l_k\cdot((0,0),f_k,(0,0))=d_0(l_k)f_{1,k}+d_1(l_k)f_{2,k},\notag\\
l_k\cdot f_{2,k}&=l_k\cdot((0,0),(0,0),f_k)=d_0(l_k)f_{2,k}.\notag
\end{align} 
Hence, 
\begin{align} 
\Delta(e_0)&=d_0(l_1)f_{0,1}+d_1(l_1)f_{1,1}+d_{2}(l_1)f_{2,1}+d_0(l_2)f_{0,2}+d_1(l_2)f_{1,2}+d_{2}(l_2)f_{2,2},\notag\\
\Delta(e_1)&=d_0(l_1)f_{1,1}+d_1(l_1)f_{2,1}+d_0(l_2)f_{1,2}+d_1(l_2)f_{2,2},\notag\\
\Delta(e_2)&=d_0(l_1)f_{2,1}+d_0(l_2)f_{2,2}.\notag
\end{align} 
The coefficients with respect to the ordered basis $\{f_{0,1},f_{0,2},f_{1,1},f_{1,2},f_{2,1},f_{2,2}\}$ give place to the following matrix representation of $\Delta$:
\begin{equation}\notag
\Delta=
\left( \begin{array}{ccc} 
d_0(l_1) 	& 0 			& 0 \\
d_0(l_2) 	& 0			& 0 \\
d_1(l_1)  & d_0(l_1) 	& 0 \\
d_1(l_2) 	& d_0(l_2) 	& 0\\
d_2(l_1) 	& d_1(l_1) 	& d_{0}(l_1)\\
d_2(l_2) 	& d_1(l_2) 	& d_{0}(l_2)
\end{array} \right)=D_2(L)^t,
\end{equation}
\end{exam}

\begin{proof}[Proof of Theorem \ref{presentation}]
We first find a presentation of $F\otimes_A B_n$ as an $A$-module. Since tensoring is right-exact, from the presentation of $F$ we obtain
\begin{equation}\label{s1}
\xymatrix{A^{b}\otimes_A B_n\ar[r]^{\overline{\Delta}}&A^{a}\otimes_A B_n \ar[r]^{\overline{\theta}}&F\otimes_A B_n\ar[r]&0},
\end{equation}
where $\overline{\Delta}:=L^t\otimes Id_{B_n}$ and $\overline{\theta}:=\theta\otimes Id_{B_n}$. Notice that $\overline{\Delta}$ and $\overline{\theta}$ are also $A_n$-module homomorphisms (with the $A_n$-module structure induced by that of $B_n$). Therefore, the exact sequence (\ref{s1}) is also an exact sequence of $A_n$-modules. Using the isomorphisms (\ref{iso tens}), the sequence (\ref{s1}) induces the following exact sequence:
\[\xymatrix{\big(A_n^{b}\big)^{n+1}\ar[r]^{{\Delta}}&\big(A_n^{a}\big)^{n+1} \ar[r]&F\otimes_A B_n\ar[r]&0}.\]

Now we describe explicitly $\Delta$ to verify that it coincides with $D_n(L)$. Denote as $\Phi_1:A^{b}\otimes_AB_n\to (A_n^{b})^{n+1}$ and $\Phi_2:A^{a}\otimes_A B_n\to (A_n^{a})^{n+1}$, the isomorphisms (\ref{iso tens}). Let $\{e_{j}\otimes t^i\}_{j,i}$ and $\{f_{k}\otimes t^i\}_{k,i}$ be the canonical bases of $A^{b}\otimes_AB_n$ and  $A^{a}\otimes_A B_n$, respectively. Similarly, $\{e_{i,j}\}_{i,j}$ and $\{f_{i,k}\}_{i,k}$ for  $(A_n^{b})^{n+1}$ and $(A_n^{a})^{n+1}$, respectively.  By (\ref{iso tens}), we have $\Phi_1(e_{j}\otimes t^i)=e_{i,j}$ and $\Phi_2(f_{k}\otimes t^i)=f_{i,k}$.

By construction, $\Delta\circ\Phi_1=\Phi_2\circ\overline{\Delta}$. Let $i\in\{0,\ldots,n\}$ and $j\in\{1,\ldots,b\}$. Denote as $L_j$ the $j$th row of $L$. Thus,
\[\begin{split}
\Delta(e_{i,j})&=\Delta(\Phi_1(e_{j}\otimes t^i))\\
&=\Phi_2(\overline{\Delta}(e_{j}\otimes t^i))\\
&=\Phi_2(L_j\otimes t^i)\\
&=\sum_{k=1}^a (L_j)_{k}\cdot f_{i,k}.
\end{split}\]
Let us analyze the product $(L_j)_{k}\cdot f_{i,k}$, for a fixed $k$. Using the $A$-module structure on $(A_n^{a})^{n+1}$ described in (\ref{struct A mod}), we obtain:
\[\begin{split}
(L_j)_k\cdot f_{i,k}&=(L_j)_k\cdot(\cc,\ldots,f_k,\ldots,\cc)\\
&=d_0((L_j)_k)f_{i,k}+d_1((L_j)_k)f_{i+1,k}+\cdots+d_{n-i}((L_j)_k)f_{n,k}.
\end{split}\]
Therefore, $\Delta(e_{i,j})=\sum_{k=1}^a(L_j)_k\cdot f_{i,k}=\sum_{k=1}^a\sum_{l=0}^{n-i}d_l((L_j)_k)f_{l+i,k}$. In particular, the $(i,j)$-th column of $\Delta$ is:
$$\big(\cc,\dots,\cc,d_{0}(L_j),\dots,d_{n-i}(L_j)\big).$$
This is precisely the $(i,j)$-th column of $D_n(L)$.
\end{proof}

The following corollary shows that a certain homological property of $F$ is inherited by $F\otimes_A B_n$.

\begin{coro}\label{gral pd}
Suppose there is an exact sequence
$$\xymatrix{0\ar[r]&A^{b}\ar[r]^{L^t}&A^{a}\ar[r]&F\ar[r]&0}.$$
If $A_n$ is a domain then $\pd_{A_n}(F\otimes_A B_n)\leq1$, where $\pd(\cdot)$ denotes the projective dimension.
\end{coro}
\begin{proof}
By the hypothesis, there is a $(b\times b)$-submatrix of $L^t$ having non-zero determinant. By its shape, it follows that there is a $b(n+1)\times b(n+1)$-submatrix of $D_n(L)$ whose determinant is non-zero. Since $A_n$ is a domain, Theorem \ref{presentation} gives a free resolution of $F\otimes_A B_n$ of length 1. 
\end{proof}


\subsection{The module of K\"ahler differentials}\label{F=Ome}

In this section we apply the previous general results to the particular case of the module of K\"ahler differentials of finitely generated algebras. 

Let $\K$ be a field and $A=\pols/\langle f_1,\ldots,f_r \rangle$. Let $\Ome$ be the module of K\"ahler differentials of $A$ over $\K$. Denote as $\Jac(f)$ the Jacobian matrix of the polynomials $f_i's$. The following presentation is well-known, 

$$\xymatrix{A^{r}\ar[rr]^{\Jac(f)^t}&&A^{s} \ar[r]&\Ome\ar[r]&0}.$$

Thus, Theorem \ref{presentation} gives:

\begin{coro}\label{pres om1}
The $A_n$-module $\Ome\otimes_AB_n$ has the following presentation:
\[\xymatrix{\big(A_n^{r}\big)^{n+1}\ar[rr]^{D_n(\Jac(f))^t}&&\big(A_n^{s}\big)^{n+1} \ar[r]&\Ome\otimes_A B_n\ar[r]&0}.\]
\end{coro}

In the following corollary we show that the matrix $D_n(\Jac(f))$ coincides with the Jacobian matrix of $d_0(f),\dots,d_n(f)$. Notice that the derivatives in $\Jac(f)$ are taken with respect to the variables $x_1,\ldots,x_s$. On the other hand, since $d_k(f_j)\in \pols_n=\K[x_i^{(0)},\ldots,x_i^{(n)}]_{i=1,\dots,s}$ \cite[Proposition 5.1]{V}, the derivatives in $\Jac(d_0(f),\dots,d_n(f))$ are taken with respect to the variables $\{x_i^{(0)},\ldots,x_i^{(n)}\}_{i=1,\dots,s}$.

\begin{coro}\label{fdbd}
Let $A$ be as before. Then 
$$D_n(\Jac(f))=\Jac(d_0(f),\dots,d_n(f)).$$ 
In particular, $\Omen\cong\Ome\otimes_A B_n$ as $A_n$-modules.
\end{coro}
\begin{proof}
The equality of the matrices is a consequence of the following known identities, where $0\leq j\leq k\leq n$ (see lemma \ref{comp der 1} and proposition \ref{comp deriv} below),
$$\frac{\partial}{\partial x_i^{(j)}}(d_k(f_l))=d_{k-j}\Big(\frac{\partial}{\partial x_i}(f_l)\Big).$$
In addition, $A_n=\frac{\K[x_i^{(0)},\ldots,x_i^{(n)}]_{i=1,\dots,s}}{\langle d_0(f_j),\dots, d_n(f_j)|j=1,\dots,r \rangle}$ \cite[Corollary 5.3]{V}. By corollary \ref{pres om1} and what we have just proved, we conclude that $\Omega_{A_n/k}^{1}$ and $\Omega_{A/k}^1\otimes_A B_n$ have identical presentations. The result follows.
\end{proof}

\begin{rem}\label{iso FD}
The isomorphism of corollary \ref{fdbd} holds for arbitrary algebras. This was initially proved in \cite[Theorem 5.3]{dFDo1} and was later revisited in \cite[Theorem 3.11]{CN}. The techniques leading to those theorems are quite different from the ones presented here.
\end{rem}

The identities appearing in the proof of corollary \ref{fdbd} seem to be known. However, we could not find any precise reference. For the sake of completeness, we provide a proof of those identities, starting with the polynomial ring in one variable.

\begin{lem}\label{comp der 1}
Let $j,k\in\N$, $0\leq j \leq k \leq n$. Then 
\[\partial_{x^{(j)}}\circ d_k=d_{k-j}\circ \partial_x.\] 
\end{lem}

\begin{proof}
It is enough to check the identity on monomials $x^i$, $i\geq1$. We proceed by induction on $k$. For $k=0$ we have that $j=0$. Then,
\[\begin{split}
\partial_{x^{(0)}}\big(d_0(x^i)\big)=&\partial_{x^{(0)}}\big((x^i)^{(0)}\big)=\partial_{x^{(0)}}\big((x^{(0)})^i\big)\\
=&i\big((x^{(0)})^{i-1}\big)=d_0\big(ix^{i-1}\big)=d_0\big(\partial_x(x^i)\big).
\end{split}\]
Now, assume that $\partial_{x^{(j)}}\circ d_l=d_{l-j}\circ \partial_x,$ for $0\leq j \leq l \leq k-1$. We show that:
\[
\partial_{x^{(j)}}\circ d_k=d_{k-j}\circ \partial_x.
\]
To prove this identity we use induction on $i$. Let $i=1$ and observe that:
 \begin{equation}\notag
\partial_{x^{(j)}}\big(x^{(k)}\big)=
\left\{ \begin{array}{ll} 
0 & \textnormal{ if } k\neq j \\
1  & \textnormal{ if } k=j \\
\end{array} \right.,
\end{equation}
\begin{equation}\notag
d_{k-j}\big(\partial_{x}(x)\big)=
\left\{ \begin{array}{ll} 
0 & \textnormal{ if } k\neq j \\
1  & \textnormal{ if } k=j \\
\end{array} \right..
\end{equation}
Assume that $\partial_{x^{(j)}}\big(d_k(x^t)\big)=d_{k-j}\big(\partial_x(x^t)\big),1\leq t\leq i-1.$ Set $0\leq j\leq k$ and let $x^i\in\K[x].$  Applying the properties of HS derivations and the Leibniz rule we obtain:
\[\begin{split}
\partial_{x^{(j)}}\big(d_k(x^i)\big)=&\partial_{x^{(j)}}\big(d_k(xx^{i-1})\big)=\partial_{x^{(j)}}\big(\sum_{l=0}^kx^{(k-l)}(x^{i-1})^{(l)}\big)\\
=&\sum_{l=0}^k\Big((x ^{i-1})^{(l)}\partial_{x^{(j)}}\big(x^{(k-l)}\big)+x^{(k-l)}\partial_{x^{(j)}}\big((x^{i-1})^{(l)}\big)\Big)\\
=&(x ^{i-1})^{(k-j)}\partial_{x^{(j)}}\big(x^{(j)}\big)+\sum_{l=0}^k x^{(k-l)}\partial_{x^{(j)}}\big((x^{i-1})^{(l)}\big)\\
=&(x ^{i-1})^{(k-j)}+\sum_{l=j}^k x^{(k-l)}\partial_{x^{(j)}}\big((x^{i-1})^{(l)}\big)\\
\end{split}\]
We have two cases: $k-j<j$ and $k-j\geq j.$ First suppose that $k-j<j$. Applying the induction hypothesis on $k$ and $i$ we get
\[\begin{split}
\partial_{x^{(j)}}\big(d_k(x^i)\big)&=(x ^{i-1})^{(k-j)}+\sum_{l=j}^kx^{(k-l)}\partial_{x^{(j)}}\big((x^{i-1})^{(l)}\big)\\
=&(x ^{i-1})^{(k-j)}+\sum_{l=j}^kx^{(k-l)}d_{l-j}\big((i-1)(x^{i-2})\big)\\
=&(x ^{i-1})^{(k-j)}+(i-1)\sum_{l=j}^k\sum_{j_1+\ldots+j_{i-2}=l-j}x^{(k-l)}x^{(j_1)}\cdots x^{(j_{i-2})}\\
=&(x ^{i-1})^{(k-j)}+(i-1)\sum_{r=0}^{k-j}\sum_{r+j_1+\ldots+j_{i-2}=k-j}x^{(r)}x^{(j_1)}\cdots x^{(j_{i-2})}\\
=&(x ^{i-1})^{(k-j)}+(i-1)(x ^{i-1})^{(k-j)}\\
=&(1+i-1)(x ^{i-1})^{(k-j)}=(ix ^{i-1})^{(k-j)}=d_{k-j}(\partial_x(x^i)).
\end{split}\]
Now, suppose that $k-j\geq j.$ We obtain
\[\begin{split}
\partial_{x^{(j)}}\big(d_k(x^i)\big)=&(x ^{i-1})^{(k-j)}+\sum_{l=j}^kx^{(k-l)}\partial_{x^{(j)}}\big((x^{i-1})^{(l)}\big)\\
=&(x ^{i-1})^{(k-j)}+x^{(j)}\partial_{x^{(j)}}\big((x ^{i-1})^{(k-j)}\big)\\
&+\sum_{\substack{l=j\\l\neq k-j}}^kx^{(k-l)}\partial_{x^{(j)}}\big((x^{i-1})^{(l)}\big).
\end{split}\]
Applying the induction hypothesis on $k$ and $i$ we get:
\[\begin{split}
\partial_{x^{(j)}}\big(d_k(x^i)\big)=&(x ^{i-1})^{(k-j)}+x^{(j)}d_{k-2j}\big((i-1)x ^{i-2}\big)\\
+&\sum_{\substack{l=j\\l\neq k-j}}^kx^{(k-l)}d_{l-j}\big((i-1)x^{i-2}\big)\\
=&(x ^{i-1})^{(k-j)}+(i-1)\sum_{j_1+\ldots+j_{i-2}=k-2j}x^{(j)}x^{(j_1)}\cdots x^{(j_{i-2})}\\
+&(i-1)\sum_{\substack{l=j\\l\neq k-j}}^k\sum_{j_1+\ldots+j_{i-2}=l-j}x^{(k-l)}x^{(j_1)}\cdots x^{(j_{i-2})}\\
=&(x ^{i-1})^{(k-j)}+(i-1)\sum_{j+j_1+\ldots+j_{i-2}=k-j}x^{(j)}x^{(j_1)}\cdots x^{(j_{i-2})}\\
+&(i-1)\sum_{\substack{r=0\\r\neq j}}^{k-j}\sum_{r+j_1+\ldots+j_{i-2}=k-j}x^{(r)}x^{(j_1)}\cdots x^{(j_{i-2})}\\
=&(x ^{i-1})^{(k-j)}+(i-1)\Big(\sum_{j+j_1+\ldots+j_{i-2}=k-j}x^{(j)}x^{(j_1)}\cdots x^{(j_{i-2})}\\
+&\sum_{\substack{r=0\\r\neq j}}^{k-j}\sum_{r+j_1+\ldots+j_{i-2}=k-j}x^{(r)}x^{(j_1)}\cdots x^{(j_{i-2})}\Big)\\
=&(x ^{i-1})^{(k-j)}+(i-1)\Big(\sum_{r=0}^{k-j}\sum_{r+j_1+\ldots+j_{i-2}=k-j}x^{(r)}x^{(j_1)}\cdots x^{(j_{i-2})}\Big)\\
=&(x ^{i-1})^{(k-j)}+(i-1)(x ^{i-1})^{(k-j)}\\
=&(1+i-1)(x ^{i-1})^{(k-j)}\\
=&(ix ^{i-1})^{(k-j)}\\
=&d_{k-j}(\partial_x(x^i)).
\end{split}\]
\end{proof}

Now we can prove the identity used in the proof of corollary \ref{fdbd}.

\begin{pro}\label{comp deriv}
Let $k\in\{0,\ldots,n\}$, $j\in\{0,\ldots,k\}$ and $i\in\{1,\ldots,s\}$. Then 
\[\partial_{x_i^{(j)}}\circ d_{k}=d_{k-j}\circ \partial_{x_i}.\]
\end{pro}
\begin{proof}
It is enough to check the desired equality on monomials. Let us denote $x_i^{\alpha_i}X^{\alpha}:=x_i^{\alpha_i}x_1^{\alpha_1}\cdots x_{i-1}^{\alpha_{i-1}}x_{i+1}^{\alpha_{i+1}}\cdots x_s^{\alpha_s}\in \K[x_1\dots,x_s]$.

\[\begin{split}
\partial_{x_i^{(j)}}(d_k(x_i^{\alpha_i}X^{\alpha}))
&=\partial_{x_i^{(j)}}\big(\sum_{l=0}^k(x_i^{\alpha_i})^{(l)}(X^{\alpha})^{(k-l)}\big)\\
&=\sum_{l=0}^k\partial_{x_i^{(j)}}\big((x_i^{\alpha_i})^{(l)}(X^{\alpha})^{(k-l)}\big)\\
&=\sum_{l=0}^k(X^{\alpha})^{(k-l)}\partial_{x_i^{(j)}}\big((x_i^{\alpha_i})^{(l)}\big)\\
&=\sum_{l=j}^k(X^{\alpha})^{(k-l)}\partial_{x_i^{(j)}}\big((x_i^{\alpha_i})^{(l)}\big).
\end{split}\]
By lemma \ref{comp der 1}, we have
\[
\begin{split}
\sum_{l=j}^j(X^{\alpha})^{(k-l)}\partial_{x_i^{(j)}}\big((x_i^{\alpha_i})^{(l)}\big)
&=\sum_{l=j}^k(X^{\alpha})^{(k-l)}d_{l-j}(\partial_{x_i}(x_i^{\alpha_i}))\\
&=\sum_{l=j}^k(X^{\alpha})^{(k-l)}\alpha_id_{l-j}(x_i^{\alpha_i-1})\\
&=\alpha_i\sum_{l'=0}^{k-j}(X^{\alpha})^{(k-j-l')}(x_i^{\alpha_i-1})^{(l')}\\
&=\alpha_id_{k-j}\big(X^{\alpha}x_i^{\alpha_i-1}\big)=d_{k-j}\big(X^{\alpha}\alpha_ix_i^{\alpha_i-1}\big)\\
&=d_{k-j}\big(\partial_{x_i}(x_i^{\alpha_i}X^{\alpha}))\big).
\end{split}
\]
\end{proof}

Next we explore corollary \ref{gral pd} in this context. 

\begin{coro}\label{pd}
Let $A$ be as before. Assume, in addition, that $A$ is a domain and a complete intersection. If $A_n$ is a domain then $\pd_{A_n}\big(\Omega^1_{A_n/\K}\big)\leq1.$
\end{coro}
\begin{proof}
By the hypothesis, we have the following exact sequence:
$$\xymatrix{0\ar[r]&A^{r}\ar[rr]^{\Jac(f)^t}&&A^{s} \ar[r]&\Ome\ar[r]&0}.$$
The result follows from corollaries \ref{gral pd} and \ref{fdbd}.
\end{proof}

M. Mustata proved that if the $n$-th jet scheme of a complete intersection is irreducible then it is also a complete intersection \cite[Proposition 1.4]{M}. Using the previous corollary we obtain the same result in a special case.

\begin{coro}\label{CI}
Consider the hypothesis of corollary \ref{pd}. Suppose that for each maximal ideal $\m\in \Sp(A_n)$, the extension $\K\subset\K(\m)$ is separable. Then $A_n$ is a complete intersection.
\end{coro}
\begin{proof}
Under these conditions, $A_n$ is a complete intersection if and only if $\pd_{A_n}\big(\Omega^{1}_{A_n/\K}\big)\leq1$, \cite[Th\'eor\`eme 2]{F}. The result follows from the previous corollary.
\end{proof}


\subsection{The module of high-order K\"ahler differentials}\label{omegam}

In this section we apply the presentation of Theorem \ref{presentation} to the particular case of the module of high-order differentials. We first recall the definition of this module. 

Throughout this section we fix $m,s\in\N$ and denote 
\[
\begin{split}
&\Lsc:=\{\beta\in\N^s|0\leq|\beta|\leq m-1\},\\
&\Ls:=\{\alpha\in\N^s|1\leq|\alpha|\leq m\}.
\end{split}
\]

Let $\K$ be a ring and $A$ be a $\K$-algebra. Denote $I_A:=\ker(A\otimes_{\K}A\rightarrow A,\mbox{ }a\otimes b\mapsto ab)$. Giving structure of $A$-module to $A\otimes_{\K}A$ by multiplying on the left, we define the $A$-module $$\Omega^{(m)}_{A/\K}:=I_A/I_A^{m+1}.$$ 

\begin{defi}\cite[Chapter II-1]{N}
The $A$-module $\Omega^{(m)}_{A/k}$ is called the \textit{module of K\"ahler differentials of order} $m\geq1$ of $A$ over $\K$. For $m=1$ this is the usual module of K\"ahler differentials.
\end{defi}

The module of high-order differentials of a finitely generated $\K$-algebra also has a presentation given by the so-called higher-order Jacobian matrix.

\begin{defi}\cite{D2,BD,BJNB}\label{Jacm}
Let $\K$ be a field and $f\in\pols$. Denote
$$\Jac_m(f):=\Big(\frac{1}{(\alpha-\beta)!}\frac{\partial^{\alpha-\beta}(f)}{\partial x^{\alpha-\beta}}\Big)_{\substack{\beta\in\Lsc \\ \alpha\in\Ls}},$$
where we define $\frac{1}{(\alpha-\beta)!}\frac{\partial^{\alpha-\beta}(f)}{\partial x^{\alpha-\beta}}=0,$ whenever $\alpha_i<\beta_i$ for some $i$. Now let $f=(f_1,\ldots,f_r)\in(\pols)^r$. Denote 
\[\Jac_{m}(f):=
\left( 
\begin{array}{cccc}
\Jac_{m}(f_1) \\
\Jac_{m}(f_2)\\
\vdots \\
\Jac_{m}(f_r)
\end{array} 
\right).\]
We call $\Jac_{m}(f)$ the \textit{Jacobian matrix of order} $m\geq1$ of $f_1,\ldots,f_r$. It is a $(rM\times N)$-matrix, where $M=\binom{m+s-1}{s}$ and $N=\binom{m+s}{s}-1$. 
\end{defi}

\begin{exam}\label{e3-1}
Let $f=x^3-y^2\in \K[x,y]$. Then 
\[\Jac_2(f)=
\left( 
\begin{array}{ccccc}
             3x^2& -2y & 3x & 0 & -1 \\
             f & 0 & 3x^2 & -2y & 0 \\
             0 & f & 0 & 3x^2 & -2y\\
\end{array} 
\right).\] 
\end{exam}

\begin{rem}
The higher-order Jacobian matrix has proved useful in the study of properties of $\Omega^{(m)}$ and of the higher Nash blowup \cite{BD,BJNB,ChDG,D2,Ch}.
\end{rem}

The Jacobian matrix of order $m$ has the following analogous property of the usual Jacobian matrix.

\begin{teo}\cite[Theorem 2.8]{BD},\cite[Corollary 2.27]{BJNB}\label{pres om}
Let $A$ denote the quotient $\pols/\langle f_1,\ldots,f_r \rangle$. The $A$-module $\OmeA$ has the following presentation,
$$\xymatrix{A^{rM}\ar[rr]^{\Jac_m(f)^t}&&A^{N} \ar[r]&\OmeA\ar[r]&0}.$$
\end{teo}

Thus, Theorem \ref{presentation} gives,

\begin{coro}\label{pres ome}
The $A_n$-module $\OmeA\otimes_A B_n$ has the following presentation:
\[\xymatrix{\big(A_n^{rM}\big)^{n+1}\ar[rrr]^{D_n(\Jac_m(f))^t}&&&\big(A_n^{N}\big)^{n+1} \ar[r]&\OmeA\otimes B_n\ar[r]&0}.\]
\end{coro}

In addition, we have the following consequence in the case of hypersurfaces.

\begin{coro}\label{pres ome hyper}
Let $f\in\pols$ be irreducible, $A=\pols/\langle f \rangle.$ If $A_n$ is a domain then $\pd_{A_n}\big(\OmeA\otimes_AB_n\big)\leq 1.$
\end{coro}
\begin{proof}
We have the following exact sequence \cite[Theorem 3.6]{BD}:
$$\xymatrix{0\ar[r]&A^{M}\ar[rr]^{\Jac_m(f)^t}&&A^{N} \ar[r]&\OmeA\ar[r]&0}.$$
The result follows from corollary \ref{gral pd}.
\end{proof}

\begin{rem}\label{no isom}
It is tempting to ask whether the isomorphism of corollary \ref{fdbd} also holds for the module of high-order differentials. Unfortunately, a straightforward computation of ranks shows that the desired isomorphism is false for $m>1$. For instance, if $A=\K[x]$, $n=1$, $m=2$, then $\Omega^{(2)}_{A_1}$ is free of rank 5 but $\Omega^{(2)}_{A}\otimes B_1$ is free of rank 4.
\end{rem}


\section{A Nobile-like theorem}

In this section we study the Nash transformation of $\OB$ in the case of hypersurfaces. We show that it shares some nice properties of Nash blowups. To that effect, we use a method proposed by O. Villamayor to explicitly compute these blowups in terms of presentations \cite{Vi}. All varieties are considered over an algebraically closed field $\K$ of arbitrary characteristic.

\begin{lem}\label{jac crit}
Let $X=\V(f)\subset\A^{s}_{\K}$ be an irreducible hypersurface and assume that $J_n(X)\subset \A^{s(n+1)}_{\K}$ is also irreducible, for some $n$. Consider a point $p=(x_1,\ldots,x_s,x_1^{(1)},\ldots,x_s^{(1)},\ldots)\in J_n(X)$. Then $p$ is non-singular if and only if $\rk D_n(\Jac_m(f)|_{p})=(n+1)M$.
\end{lem}
\begin{proof}
Firstly, $J_n(X)$ irreducible implies $\dim J_n(X)=(s-1)(n+1)$. Assume that $p$ is non-singular. Recall that $J_n(X)=\V(f,d_1(f),\ldots,d_n(f))$. The Jacobian matrix of $f,d_1(f),\ldots,d_n(f)$ has the following shape:
\[
\left( 
\begin{array}{cccccccc}
\partial_{x_1}f 		& \cdots 		& \partial_{x_s}f 				& 0 							&\cdots \\
\partial_{x_1}d_1(f) 	& \cdots        &\partial_{x_s}d_1(f)			& \partial_{x_1^{(1)}}d_1(f) 	& \cdots 		& \partial_{x_s^{(1)}}d_1(f) 	&0							&\cdots\\
\partial_{x_1}d_2(f) 	& \cdots       	&\partial_{x_s}d_2(f)			&\partial_{x_1^{(1)}}d_2(f) 	& \cdots 		& \partial_{x_s^{(1)}}d_2(f) 	&\partial_{x_1^{(2)}}d_2(f)		&\cdots\\
\vdots 				& 	 		& \vdots 						& \vdots 				\\		
\partial_{x_1}d_n(f) 	& \cdots        &\partial_{x_s}d_n(f)			& \partial_{x_1^{(1)}}d_n(f) 	& \cdots 		& \partial_{x_s^{(1)}}d_n(f) 	&\cdots
\end{array} 
\right).
\] 
By the usual Jacobian criterion we have $$\rk(\Jac(f,d_1(f),\ldots,d_n(f))|_p)=s(n+1)-(s-1)(n+1)=n+1.$$
In particular, $\partial_{x_j}f(p)\neq0$ for some $j$. After renaming the variables, if necessary, we may assume $j=1$. It is known that, in this case, $\Jac_m(f)$ can be ordered so that it has row echelon form with $\partial_{x_1}(f)$ as pivots \cite[Lemma 2.11]{BD}. Then, by the shape of $D_n(\Jac_{m}(f))$, we obtain $\rk D_n(\Jac_{m}(f))|_p=M(n+1)$.

Now suppose that $\rk D_n(\Jac_{m}(f))|_p=M(n+1)$. Again by the shape of this matrix, it follows that $\partial_{x_j}f(x_1,\ldots,x_s)=\partial_{x_j}f(p)\neq0$, for some $j$. Then $(x_1,\ldots,x_s)\in X$ is non singular. Thus, $p\in J_n(X)$ is also non singular.
\end{proof}

\begin{rem}
Notice that the proof of the previous lemma shows a little more: the statement is also equivalent to $X$ being non-singular at $q=\pi_n(p)$, where $\pi_n:J_n(X)\to X$ is the natural projection. Indeed, if $p$ is non-singular it was shown that $\partial_{x_j}f(q)=\partial_{x_j}f(p)\neq0$ for some $j$. Hence $q$ is non-singular. On the other hand, it is known that $q$ non-singular implies $p$ is non-singular.
\end{rem}

The following theorem can be seen as a generalization, in a special case, of a well-known theorem due to A. Nobile \cite{No}. There are other generalizations of Nobile's theorem \cite{D1,D2,ChDG,DN1,DN2}.

\begin{teo}\label{Nobile}
Let $X\subset\A^s_{\K}$ be an irreducible singular hypersurface and assume that $J_n(X)$ is irreducible and normal. Then the blowup of $J_n(X)$ at $\OmeA\otimes_AB_n$ is not an isomorphism.
\end{teo}
\begin{proof}
Let $A$ be the coordinate ring of $X$. We have an exact sequence \cite[Theorem 3.6]{BD}:
$$\xymatrix{0\ar[r]&A^{M}\ar[rr]^{\Jac_m(f)^t}&&A^{N} \ar[r]&\OmeA\ar[r]&0}.$$
By the proof of corollary \ref{gral pd} we obtain the following exact sequence:
\[\xymatrix{0\ar[r]&\big(A_n^{M}\big)^{n+1}\ar[rrr]^{D_n(\Jac_m(f))^t}&&&\big(A_n^{N}\big)^{n+1} \ar[r]&\OmeA\otimes B_n\ar[r]&0}.\]
Denote as $\J_{n,m}$ the ideal generated by all $(n+1)M$-minors of $\Dnm$. Let $Q=\mbox{Frac}(A_n)$. Tensoring this exact sequence with $Q$, it follows that $\dim_Q(\OmeA\otimes_A B_n)\otimes_{A_n}Q=(n+1)(N-M)$. In particular, the blowup of $J_n(X)$ at $\OmeA\otimes_A B_n$ coincides with the blowup of the ideal $\J_{n,m}$ \cite[Proposition 2.5]{Vi}. 

On the other hand, by lemma \ref{jac crit}, $\V(\J_{n,m})=\Sing(J_n(X))$. Since $X$ is singular, $J_n(X)$ is singular \cite[Corollary 1.2]{I}. It follows that $\J_{n,m}$ is not locally principal. Hence its blowup is not an isomorphism.
\end{proof}

\begin{rem}
An interesting situation where the hypothesis of the previous theorem are satisfied is the case of terminal hypersurface singularities \cite[Theorem 0.2]{EMY}.
\end{rem}

For $n=0$, the previous result recovers one the main theorems of \cite{DN1}.

\begin{coro}\cite[Theorem 4.2]{DN1}
Let $X$ be a normal and irreducible hypersurface. If the higher Nash blowup of $X$ is an isomorphism then $X$ is non-singular.
\end{coro}

We conclude with a comment regarding Nash blowups of jet schemes.

\begin{rem}
Recall the notation of Theorem \ref{FD Nash}. It was proved by T. de Fernex and R. Docampo that $N(J_n'(X),(\Omega^1_{X/\K})_n')$ is isomorphic to the Nash blowup of $J_n'(X)$ (this is a direct consequence of the isomorphism $\Omega^1_{A_n}\cong \Omega^1_A\otimes B_n$, see remark \ref{iso FD}). In particular, for $m=1$, Theorem \ref{Nobile} follows from \cite{No,DN1}. Because of remark \ref{no isom}, it is not clear whether $N(J_n'(X),(\Omega^{(m)}_{X/\K})_n')$ is isomorphic to the higher Nash blowup of $J_n'(X)$. Theorem \ref{Nobile} can be seen as a positive evidence towards this question.
\end{rem}

\section*{Acknowledgements}

We thank Roi Docampo and Takehiko Yasuda for their motivating comments and questions on a previous version of this paper. We also thank the referee for the careful reading, the nice remarks, and for drawing our attention to the reference \cite{EMY}.

\vspace{.5cm}
\noindent{\footnotesize \textsc {Paul Barajas, Universidad Aut\'onoma de Zacatecas, Paseo a La Bufa entronque Solidaridad s/n, CP 98000, Zacatecas, Mexico} \\
36178329@uaz.edu.mx, paulvbg@gmail.com}\\
{\footnotesize \textsc {Daniel Duarte, Universidad Aut\'onoma de Zacatecas-CONACyT, Paseo a La Bufa entronque Solidaridad s/n, CP 98000, Zacatecas, Mexico} \\
aduarte@uaz.edu.mx}\\
\end{document}